\documentclass{article}

\usepackage{amsmath, amsthm, amssymb}
\usepackage{authblk}

\newtheorem{theorem}{Theorem}[section]
\newtheorem{dnt}[theorem]{Definition}

\newtheorem{lemma}[theorem]{Lemma}
\newtheorem{prop}[theorem]{Proposition}
\newtheorem{cor}[theorem]{Corollary}
\newtheorem{prb}[theorem]{Problem}

\begin{document}
\title{Daisy Hamming graphs}
\author[1,2]{Tanja Gologranc\thanks{tanja.gologranc1@um.si}}
\author[1,2]{Andrej Taranenko\thanks{andrej.taranenko@um.si}}

\affil[1]{University of Maribor, Faculty of Natural Sciences and Mathematics, Koro\v{s}ka cesta 160, SI-2000 Maribor, Slovenia}
\affil[2]{Institute of Mathematics, Physics and Mechanics, Jadranska 19, SI-1000 Ljubljana, Slovenia}

\maketitle

\begin{abstract}
Daisy graphs of a rooted graph $G$ with the root $r$ were recently introduced as a generalization of daisy cubes, a class of isometric subgraphs of hypercubes. In this paper we first solve the problem posed in~\cite{Taranenko2020} and characterize rooted graphs $G$ with the root $r$ for which all daisy graphs of $G$ with respect to $r$ are isometric in $G$. We continue the investigation of daisy graphs $G$ (generated by $X$) of a Hamming graph $H$ and characterize those daisy graphs generated by $X$ of cardinality 2 that are isometric in $H$. Finally, we give a characterization of isometric daisy graphs of a Hamming graph $K_{k_1}\Box \ldots \Box K_{k_n}$ with respect to $0^n$ in terms of an expansion procedure.


\end{abstract}

\section{Introduction and preliminary results}
A recent paper by Klavžar and Mollard \cite{KlaMol-19} introduced a new family of graphs called daisy cubes.
The daisy cube $Q_n(X)$ is the subgraph of
$Q_n$ induced by the union of the intervals $I(x, 0^n)$ over all $x \in X \subseteq V(Q_n)$.
Daisy cubes are shown to be partial cubes (i.e. isometric subgraphs of hypercubes) and include some other previously well known classes of cube-like graphs, e.g. Fibonacci cubes \cite{K2012} and Lucas cubes \cite{MCS2001, T2013}. Regarding daisy cubes, several results have already appeared in the literature. Vesel \cite{V2019} has shown that a cube-complement of a daisy cube is also a daisy cube. Moreover, daisy cubes also appear in chemical graph theory in connection with resonance graphs. Žigert Pleteršek has shown in \cite{ZP2018} that resonance graphs of the so-called kinky benzenoid systems are daisy cubes and Brezovnik et. al. \cite{BTZP2020} characterized catacondensed even ring systems of which resonance graphs are daisy cubes.

Taranenko \cite{Taranenko2020} characterized daisy cubes by means of special kind of pheripheral expansions and thus proved that daisy cubes are tree-like partial cubes~\cite{BIK2003}. In the same paper a generalization of daisy cubes to arbitrary rooted graphs was introduced. These graphs are called \emph{daisy graphs} of rooted graphs with respect to the root. A sufficient but not a necessary condition for a rooted graph $G$ in which every daisy graph of $G$ with respect to the root is isometric in $G$ was presented. We improve this result with another sufficient condition for this and also prove that both conditions together provide a characterization of such graphs. We present these and related results in Section \ref{sec:iso}. In Section \ref{sec:dhg} we focus on daisy graphs of Hamming graphs (with respect to a chosen root), called \emph{daisy Hamming graphs}. Since hypercubes are a special case of Hamming graphs and daisy cubes are a special case of daisy graphs, a natural question that arises is: what properties do isometric daisy Hamming graphs have. Studying the properties of these graphs we obtain a characterization of isometric daisy Hamming graphs in terms of a specific kind of expansion.

We continue this section with some notations and preliminary result.
All graphs $G=(V,E)$ in this paper are undirected and without loops or 
multiple edges. The {\it distance} $d_G(u,v)$ between
two vertices $u$ and $v$ is the length of a shortest $u,v$-path, and
the \emph{interval} $I_G(u,v)$ between $u$ and $v$ consists of all the
vertices on all shortest $u,v$-paths, that is, $I_G(u,v)=\{ x\in V(G):
d_G(u,x)+d_G(x,v)=d_G(u,v)\}.$ For a set $U$ of vertices of a graph $G$ we denote 
by $\left\langle U \right\rangle _G$ the subgraph of $G$ induced by the set $U$. The index $G$ may be omitted when the graph will be clear from the context.
A subgraph $H$ of $G$ is called {\it isometric} if
$d_H(u,v)=d_G(u,v)$, for all $u,v \in V(H)$.

The {\it Cartesian product}
$G=G_1\square\ldots\square G_n$ of $n$ graphs $G_1,\ldots,G_n$ has the
$n$-tuples $(x_1,\ldots,x_n)$ as its vertices (with vertex $x_i$ from
$G_i$) and an edge between two vertices $x=(x_1,\ldots,x_n)$ and
$y=(y_1,\ldots,y_n)$ if and only if, for some $i,$ the vertices $x_i$
and $y_i$ are adjacent in $G_i$, and $x_j=y_j$, for the remaining $j\ne
i$  \cite{HIK2011knjiga}. The Cartesian product of $n$ copies of $K_2$ is a {\it hypercube} or
$n$-{\it cube} $Q_n$. If all the factors in a Cartesian product are complete
graphs then $G$ is called a \emph{Hamming graph}. The Hamming graph $H=K_{k_1} \Box \ldots \Box K_{k_n}$ will be denoted by $H_{k_1,\ldots , k_n}$. Isometric subgraphs of hypercubes are called {\it partial cubes} and isometric subgraphs of Hamming graphs are called \emph{partial Hamming graphs}. Note, a tuple $(x_1,\ldots,x_n)$ may be written in a shorter form as  $x_1\ldots x_n$.

For any positive integer $n$ the set $\{1,\ldots , n\}$ is denoted by $[n]$ and the set $\{0,1,\ldots , n-1\}$ by $[n]_0$.
Let $k_1, \ldots , k_n$ be positive integers and let $V=\prod_{i=1}^n [k_i] _0$. The {\em Hamming distance}, $H(u,v)$, of two vectors $u,v \in V$ is the number of coordinates in which they differ. Note, a Hamming graph $H_{k_1,\ldots , k_n}$ is the graph with the vertex set $\prod_{i=1}^n[k_i]_0$, such that the Hamming distance and the distance function of the graph coincide. Let $v=v_1\ldots v_n \in V(H_{k_1,\ldots ,k_n})$. If $x_1\ldots x_n \in I_{H_{k_1,\ldots ,k_n}}(v, 0^n)$ then $x_i \in \{ 0,v_i \}, \textrm{ for any } i \in [n]$.

\begin{dnt}\cite{Mulder1980}
Let $G$ be a graph and $(u,v,w)$ a triple of vertices of $G$. A triple $(x,y,z)$ of vertices of $G$ is a {\em pseudo-median} of the triple $(u,v,w)$ if it satisfies all of the following conditions:
\begin{enumerate}
    \item \begin{enumerate}[(i)]
        \item there is a shortest $u,v$-path in $G$ that contains both $x$ and $y$;
        \item there is a shortest $v,w$-path in $G$ that contains both $y$ and $z$;
        \item there is a shortest $u,w$-path in $G$ that contains both $x$ and $z$;
    \end{enumerate}
    \item $d(x,y)=d(y,z)=d(x,z)$
    \item $d(x,y)$ is minimal under the first two conditions.
\end{enumerate}
The distance $d(x,y)$ is called the {\em size} of the pseudo-median $(x,y,z)$.
\end{dnt}

Pseudo-median of a triple $(u,v,w)$ of size 0, is called a {\em median} of $(u,v,w)$.
Let $G$ be a graph and $(u,v,w)$ a triple of vertices of $G$. A triple $(x,y,z)$ of vertices of $G$ is a {\em quasi-median} of the triple $(u,v,w)$ if it is a pseudo-median of $(u,v,w)$ and if $(u,v,w)$ has no pseudo-median different from $(x,y,z)$. Note that any triple $(u,v,w)$ of vertices $u=u_1\ldots u_n$, $v=v_1\ldots v_n$, $w=w_1\ldots w_n$ of a Hamming graph $H_{k_1,\ldots ,k_n}$ has a quasi-median $(x,y,z)$, that can be obtained in the following way. If $u_i, v_i$ and $w_i$ are pairwise distinct, then $x_i=u_i$, $y_i=v_i$, $z_i=w_i$. If $u_i,v_i$ and $w_i$ are not all pairwise distinct with at least two of $u_i,v_i,w_i$ equal to $p_i$, then $x_i=y_i=z_i=p_i$. The size of this quasi-median is the number of coordinates in which $u, v$ and $w$ are all distinct~\cite{Mulder1980}.

A binary expansion was first defined in~\cite{Mu1} and a generalization of binary expansion using more covering sets was first introduced in~\cite{Mulder1980}. We will use the definition of general expansion introduced by Chepoi~\cite{chepoi-88}, as follows.

\begin{dnt}\cite{chepoi-88}\label{def:exp}
Let $G$ be a connected graph and let $W_1,W_2, . . . ,W_n$ be subsets of $V(G)$
such that:
\begin{enumerate}
\item $W_i \cap W_j \neq \emptyset $, for all $i,j \in [n]$;
\item $ \bigcup_{i=1}^n W_i = V(G)$;
\item there are no edges between sets $W_i \setminus W_j$ and $W_j \setminus W_i$, for all $i,j \in [n]$;
\item subgraphs $\left\langle W_i \right\rangle, \left\langle W_i \cup W_j \right\rangle$ are isometric in $G$, for all $i,j \in [n]$.
\end{enumerate}
Then to each vertex $x \in V(G)$ we associate a set $\{i_1, i_2,\ldots , i_t\}$ of all indices
$i_j$, where $x \in W_{i_j}$. A graph $G'$ is called an expansion of $G$ relative to the sets
$W_1,W_2, \ldots ,W_n$ if it is obtained from $G$ in the following way:
\begin{enumerate}
\item replace each vertex $x$ of $G$ with a clique with vertices $x_{i_1} , x_{i_2} , \ldots , x_{i_t}$;
\item if an index $i_s$ belongs to both sets $\{i_1, \ldots ,i_t\}, \{i_1',\ldots ,i_l'\}$  corresponding
to adjacent vertices $x$ and $y$ in $G$ then let $x_{i_s}y_{i_s} \in E(G')$.
\end{enumerate}
\end{dnt}

An expansion of $G$ relative to the sets
$W_1,W_2, \ldots ,W_n$  is called {\em peripheral} if there exists $i\in [n]$ such that $W_i=V(G)$. The peripheral expansion of $G$ relative to the sets
$W_1, W_2, \ldots, W_n$ will be denoted by pe$(G;W_1, \ldots, W_n)$.

Let $G=(V,E)$ be a connected graph and $uv\in E(G)$. We define the following sets:
\begin{align*}
& W_{uv}=\{x \in V(G) \ |\  d(u,x) < d(v,x)\};\\
& U_{uv}=\{x \in W_{uv}\ |\  \text{there exists } z \in W_{vu} \text{ such that } xz \in E(G) \};\\
& F_{uv}=\{xz \in E(G)\ |\  x \in U_{uv} \land z \in U_{vu}\}.
\end{align*}

With these sets we can define Djokovi\' c relation $\sim$ as follows~\cite{Djokovic1973}. For $uv,xy \in E(G)$
$$uv \sim xy \textrm{ if and only if } x \in W_{uv} \land y\in W_{vu}.$$ 
It follows from the definition that $F_{uv}$ is precisely the set of edges from $E(G)$  that are in relation $\sim$ with $uv \in E(G)$. Note also that the relation $\sim$ is reflexive and symmetric but not transitive in general. In~\cite{Bresar2001} Bre\v sar introduced relation $\bigtriangleup$ on the edge set of a connected graph as follows.

\begin{dnt}\cite{Bresar2001}
Let $G$ be a connected graph and $uv, xy \in E(G)$. Then $uv \bigtriangleup xy$ if and only if $uv \sim xy$ or there exists a clique with edges $e,f \in E(G)$ such that $xy \sim e$ and $uv \sim f$.
\end{dnt}

Note that the relation $\bigtriangleup$ is also reflexive and symmetric but it is not necessarily transitive. Bre\v sar proved that the relation $\bigtriangleup$ is transitive in partial Hamming graphs~\cite{Bresar2001}. He also proved that  each $\bigtriangleup$-class is a union of some $\sim$-classes. For edges $ab,cd \in E(G)$ the $\sim$-classes $F_{ab}$ and $F_{cd}$ are in the same $\bigtriangleup$-class if and only if there is a clique containing edges $a'b' \in F_{ab}$ and $ c'd' \in F_{cd}$.

\section{Isometric daisy graphs}\label{sec:iso}

In~\cite{Taranenko2020} a generalization of daisy cubes was defined in the following way.

\begin{dnt}\cite{Taranenko2020}\label{defDG}
Let $G$ be a rooted graph with the root $r$. For $X\subseteq V(G)$ the \emph{daisy graph $G_r(X)$ of the graph $G$ with respect to $r$ (generated by $X$)} is the subgraph of $G$ where $$G_r(X)=\left\langle \{ u \in V(G)\ |\ u\in I_G(r,v) \text{ for some }v\in X\} \right\rangle.$$
\end{dnt}

If $H=G_r(X)$ is an isometric subgraph of $G$ we say that $H$ is an {\em isometric daisy graph} of a graph $G$ with respect to $r$.
Note that it follows from Definition \ref{defDG}, that $V(G_r(X))=\bigcup_{v \in X}I_G(v,r)$. Moreover, if $u \in V(G_r(X))$, then $I_G(u,r) \subseteq V(G_r(X))$. Therefore
any convex subgraph $H$ of a rooted graph $G$ with root $r$, such that $H$ contains $r$, is a daisy graph of $G$ with respect to $r$.

In~\cite{Taranenko2020} Taranenko presented a sufficient condition for a rooted graph $G$ with the root $r$ in which any daisy graph with respect to $r$ is isometric. He also proved that the mentioned condition is not necessary.

\begin{prop}~\cite{Taranenko2020}\label{medImpliesIsometric}
Let $G$ be a rooted graph with the root $r$. If for any two vertices of $G$, say $u$ and $v$, it holds that there exists a pseudo-median of $(u,v,r)$ of size 0, then every daisy graph of $G$ with respect to $r$ is isometric in $G$.
\end{prop}

We give another sufficient condition for a rooted graph $G$ with respect to the root $r$ in which any daisy graph with respect to $r$ is isometric and prove that both conditions yield a characterization of rooted graphs $G$ having all daisy graphs with respect to the root isometric.

\begin{theorem}\label{th:sufficient}
Let $G$ be a rooted graph with the root $r$. If for any two vertices of $G$, say $u$ and $v$, there exists a pseudo-median of size 1 of the triple of vertices $u$, $v$ and $r$, then every daisy graph of $G$ with respect to $r$ is isometric in $G$.
\end{theorem}

\begin{proof}
Let $H$ be an arbitrary daisy graph of $G$ with respect to $r$. Also, let $u$ and $v$ be two arbitrary vertices of $H$, and let $(x,y,z)$ be a pseudo-median of $(u,v,r)$ of size 1. Hence there exists a shortest $u,v$-path in $G$ that contains $x$ and $y$, where $x \in I_G(u,r)$ and $y \in I_G(v,r)$. Thus $I_G(u,x) \subseteq I_G(u,r) \subseteq V(H)$, as $H$ is a daisy graph of $G$ with respect to $r$ and analogously $I_G(v,y) \subseteq I_G(v,r) \subseteq V(H)$. Therefore $d_H(u,x)=d_G(u,x)$ and $d_H(v,y)=d_G(v,y)$. Since $x$ and $y$ lie on a shortest $u,v$-path we get
\begin{align*}
& d_G(u,v)=d_G(u,x)+d_G(x,y)+d_G(y,v)= \\ & =d_G(u,x)+d_G(y,v)+1=d_H(u,x)+d_H(y,v)+1 \geq d_H(u,v).    
\end{align*}
 Moreover, $H$ is a subgraph of $G$ and therefore  $d_G(u,v) \leq d_H(u,v)$ and consequently $H$ is an isometric subgraph of $G$.
\end{proof}

\begin{dnt}
A graph $G$ satisfies the triangle condition if for any three vertices $u, v, w \in V(G)$, such that $d(v,w)=1$ and $d(u,v)=d(u,w)\geq 2$ there exists a vertex $x\in V(G)$ adjacent to $v$ and $w$ with $d(x,u)=d(u,v)-1$.
\end{dnt}

\begin{dnt}
A rooted graph $G$ with the root $r$ satisfies the rooted triangle condition if for any two adjacent vertices $v, w \in V(G)$, such that $d(r,v)=d(r,w)\geq 2$ there exists a vertex $x\in V(G)$ adjacent to $v$ and $w$ with $d(x,r)=d(r,v)-1$.
\end{dnt}

\begin{theorem}\label{th:necessary}
Let $G$ be a rooted graph with the root $r$ such that $G$ satisfies the rooted triangle condition. If every daisy graph of $G$ with respect to $r$ is isometric in $G$, then for any $u,v \in V(G)$ there exists a pseudo-median in $G$ of size 0 or 1 for the triple $u,v$ and $r$.
\end{theorem}

\begin{proof}
Let $u$ and $v$ be two arbitrary vertices of a rooted graph $G$ with the root $r$. Let $H=G_r(\{u,v\})$. Hence $V(H)=I_G(u,r) \cup I_G(v,r)$. Since $H$ is an isometric subgraph of $G$, there exists a shortest $u,v$-path $P$ in $G$ which is entirely contained in $H$. Denote $P:u=u_0,u_1,\ldots , u_{k-1},u_k=v$. As $P \subseteq V(H)$, $u_i \in I_G(u,r) \cup I_G(v,r)$, for any $i \in \{0,1,\ldots , k\}$. If $v \in I_G(u,r)$, then $(v,v,v)$ is a pseudo-median of $(u,v,r)$ of size 0 and the proof is completed. If $u \in I_G(v,r)$, then $(u,u,u)$ is a pseudo-median of $(u,v,r)$ of size 0 and again, the proof is  completed. Hence we may assume that $u \notin I_G(v,r)$ and $v \notin I_G(u,r)$. Let $j \in [k]_0$ be the largest index such that $u_j \in I_G(u,r)$. Hence $u_l \in I_G(v,r)$ for any $l \in \{j+1,\ldots ,k\}$. If $u_j \in I_G(v,r)$, then $u_j \in I_G(u,r) \cap I_G(v,r)$ and hence $(u_j,u_j,u_j)$ is a pseudo-median of $(u,v,r)$ of size 0. Next, we assume that $u_j \notin I_G(v,r)$. Since $u_{j+1} \notin I_G(u,r)$, $d_G(u_j,r)=d_G(u_{j+1},r)=l$. If $l=1$, then $(u_j,u_{j+1},r)$ is a pseudo-median of $(u,v,r)$ of size 1. If $l> 1$, then by the rooted triangle condition, there exists $x \in V(G)$ that is adjacent to $u_j$ and $u_{j+1}$ and $x \in I_G(r,u_j) \cap I_G(r,u_{j+1})$. Hence $(u_j,u_{j+1},x)$ is a pseudo-median of $(u,v,r)$ of size 1, which completes the proof. 
\end{proof}

From the proof of Theorem~\ref{th:necessary} we get the following.

\begin{cor}
Let $G$ be a rooted graph with the root $r$ such that $G$ satisfies the rooted triangle condition and let $\{u,v\} \subseteq V(G)$. If $H=G_r(\{u,v\})$ is isometric in $G$, then there exists a pseudo-median in $G$ of size 0 or 1 for the triple $u,v$ and $r$.
\end{cor}

Proposition~\ref{medImpliesIsometric}, Theorem~\ref{th:sufficient} and Theorem~\ref{th:necessary} give the following characterization of rooted graphs $G$ with the root $r$ satisfying the rooted triangle condition, such that every daisy graphs of $G$ with respect to $r$ is isometric in $G$.

\begin{cor}\label{c:main}
Let $G$ be a rooted graph with the root $r$ such that $G$ satisfies the rooted triangle condition. Every daisy graph of $G$ with respect to $r$ is isometric in $G$, if and only if for any $u,v \in V(G)$ there exists a pseudo-median of size 0 or 1 of the triple of vertices $u, v$ and $r$.
\end{cor}

\begin{lemma}\label{l:TCHamming}
If $G$ is a Hamming graph, then $G$ satisfies the triangle condition.
\end{lemma}

\begin{proof}
Let $u=(u_1,\ldots, u_n)$ and $v=(v_1,\ldots, v_n)$ be two adjacent vertices of $G$ and $w=(w_1,\ldots, w_n)\in V(G)$ such that $d(u,w)=d(v,w)=k, k\geq 2.$ Since $uv\in E(G)$ there exists $i\in\{1,\ldots,n\}$ such that $v_i\not= u_i$ and $v_j=u_j$, for all $j\not= i$. Moreover, since $d(w,u)=d(w,v)=k$, it follows that $w_i\not=u_i$ and $w_i \not=v_i$. Let $x=(u_1, \ldots, u_{i-1}, w_i, u_{i+1}, \ldots, u_n)$. Clearly, $xu\in E(G)$ and $xv\in E(G)$ and $x\in I_G(u,w)\cap I_G(v,w)$. The assertion follows.
\end{proof}

Lemma~\ref{l:TCHamming} and Corollary~\ref{c:main} imply the following.

\begin{cor}
Let $G$ be a Hamming graph with the root $r$. Every daisy graph of $G$ with respect to $r$ is isometric in $G$, if and only if for any $u,v \in V(G)$ there exists a pseudo-median of size 0 or 1 for the triple $u,v$ and $r$.
\end{cor}

The above results refer to rooted graphs $G$ for which all daisy graphs with respect to the root are isometric. Now we chose one daisy graph $H$ of $G$ with respect to the root of $G$ and study when is $H$ isometric in $G$.

Note that one can easily deduce from the proofs of Proposition~\ref{medImpliesIsometric} and Theorem~\ref{th:sufficient} that if $G$ is a rooted graph with the root $r$ and $H$ a daisy graph of $G$ with respect to $r$ such that for any  $u$ and $v$ in $H$, there exists a pseudo-median of size 0 or 1 of the triple of vertices $u$, $v$ and $r$, then $H$ is isometric in $G$. It is clear that the reverse statement is not necessarily true. For example, let $G$ be the cycle $C_6$ and $u$ and $r$ two antipodal vertices of $C_6=u,x_1,x_2,r,y_1,y_2,u$. Then $G_r(\{u\})$ is the whole graph $G$ and thus isometric in $G$, but there clearly exists a triple of vertices in $G$, for example $(x_1,y_1,r)$ having no pseudo-median of size 0 or 1 in $G$. 

\begin{prb}
Let $G$ be a rooted graph with the root $r$. Characterize daisy graphs of $G$ with respect to $r$ (generated by $X$) that are isometric in $G$. 
\end{prb}

Let $G$ be a rooted graph with the root $r$. For $X=\{v\} \subseteq V(G)$ the above problem is equivalent to the characterization of intervals $I_G(v,r)$ that are isometric in $G$. 

In the rest of this section we will consider Hamming graphs and study properties of isometric daisy subgraphs. Thus let $H=H_{k_1,\ldots, k_n}$ be a Hamming graph with the root $r=0^n$. Let $G=H_r(X)$ be a daisy graph of $H$ with respect to $r$ (generated by $X$). Note that if $|X|=1$, then $G$ is a daisy cube. Moreover, if $x=x_1\ldots x_n$ is the vertex of $X$, then $G \cong Q_n(\{y_1\ldots y_n\})$, where $y_i=\min{\{x_i,1\}}$, for any $i \in \{1, \ldots , n\}$. For $|X|=2$ we have the following characterization of isometric daisy graphs of a Hamming graph. 

\begin{theorem}
Let $H=H_{k_1,\ldots ,k_n}$ be a Hamming graph with the root $0^n$ and  let $G=H_{0^n}(X)$ be a daisy graph of  $H$ generated by the set $X=\{x,y\}$ of cardinality 2. Then $G$ is an isometric subgraph of $H$ if and only if there exists a pseudo-median of $(x,y,0^n)$ of size 0 or 1 in $G$.
\end{theorem}

\begin{proof}
Let $G=H_{0^n}(\{x,y\})$. Denote $x=x_1\ldots x_n$, $y=y_1\ldots y_n$ and $r=0^n=r_1\ldots r_n$.

Suppose first, $G$ is an isometric subgraph of $H$. By Lemma~\ref{l:TCHamming} $H$ satisfies the triangle condition and consequently also the rooted triangle condition. Using the same line of thought as in the proof of Theorem~\ref{th:necessary} one can easily check that there exists a pseudo-median of $(x,y,0^n)$ of size 0 or 1 in $G$.

For the converse suppose that there is a pseudo-median of size 0 or 1 of $(x,y, 0^n)$ in $G$. Since the size of the pseudo-median in a Hamming graph is the number of coordinates in which $x,y$ and $r$ are all distinct, there is at most one coordinate in which $x,y$ and $r$ are all pairwise distinct. To simplify, permute factors of $H$ such that $x$ has the first $i-1$ coordinates equal to 0 and all other coordinates different from 0 (i.e. $i-1$ is the number of coordinates of $x$ that are equal to 0), and if there exists a coordinate in which $x,y$ and $r$ are pairwise distinct, let this be the $i^\text{th}$ coordinate. Since $(x,y,r)$ has a pseudo-median of size 0 or 1, $y_j\in \{x_j,0\}$, for any valid index $j > i$.

Let $u=u_1\ldots u_n$ and $v=v_1\ldots v_n$ be two arbitrary vertices of $G$. Note, $V(G)=I_H(x,0^n) \cup I_H(y,0^n)$. We will prove that there exists $u,v$-path in $G$ with $d_G(u,v)=H(u,v)=d_H(u,v)$.

Suppose first that $u,v \in I_H(x,0^n)$ (the case when $u,v \in I_H(y,0^n)$ is proved in a similar way). Then $u_j=v_j=0$, for any $j < i$, and for any $j \geq i$, it holds that $u_j \in \{x_j,0\}$ and $v_j \in \{x_j,0\}$. We construct $u,v$-path of length $H(u,v)$ in $G$ in the following way.  Start in $u$ and continue with $u^{(1)}$ which is obtained from $u$ by replacing the first coordinate of $u$, say $u_j$, in which $u$ and $v$ differ, by $v_j$. Since $v_j \neq u_j$ and $u,v \in I_H(x,0^n)$, $\{v_j,u_j\}=\{x_j,0\}$ and consequently $u^{(1)} \in I_H(x,0^n) \subseteq V(G)$. We continue in the same way step by step, such that at the step $k$ we replace the first coordinate of $u^{(k)}$, say $u^{(k)}_j$, in which $u^{(k)}$ and $v$ differ, by $v_j$. Since all the vertices $u^{(k)}$, for any valid $k$, are contained in $V(G)$ and the constructed path $P$ is of length $H(u,v)$, $P$ is an $u,v$-path of $G$ of length $d_H(u,v)$.  

Finally let $u \in I_H(x,0^n), v \in I_H(y,0^n)\setminus I_H(x,0^n)$. 

Let $I_D$ be the set of indices in which $u$ and $v$ differ. We will also use the following sets. The set $I_M=\{i' \in I_D\ |\ u_{i'}\neq 0 \land v_{i'}\neq 0\}$, this is an empty set, if $(x,y,r)$ has a pseudo-median of size 0, otherwise it contains the index $i$. Let $I_u = \{ i' \in I_D \ |\ u_{i'} = 0\}$ and $I_v = \{ i' \in I_D \ |\ v_{i'} = 0\}$. Note that $I_M, I_u$ and $I_v$ form a partition of $I_D$.

We construct a $u,v$-path in the following way. The first part of the path is constructed by using all the indices from the set $I_v = \{ i_1, i_2, \ldots, i_{|I_v|}\}$. Let $u^{(0)} = u$ be the first vertex of this path. The next vertex of the path, $u^{(1)}$, is obtained from $u^{(0)}$ by replacing the coordinate $u^{(0)}_{i_1}$ with 0. The vertex $u^{(2)}$, is obtained from $u^{(1)}$ by replacing the coordinate $u^{(1)}_{i_2}$ with 0. Assume we have already obtained the vertex $u^{(j)}$, then we obtain the vertex $u^{(j+1)}$ from $u^{(j)}$ by replacing the coordinate $u^{(j)}_{i_{j+1}}$ with zero. We do this for every index in $I_v$, so the last vertex we obtain is $u^{(|I_v|)}$. It is easy to see, that these vertices indeed form a path (two consecutive vertices differ in exactly one coordinate). Since we only change coordinates to 0, it is also clear that every vertex constructed so far belongs to $I_H(u, 0^n) \subseteq I_H(x, 0^n) \subseteq V(G)$.

If $I_M$ is not an empty set, we form the next vertex in our path, say $v^{(0)}$, from $u^{(|I_v|)}$ by replacing the coordinate $u^{(|I_v|)}_i$ to $v_i$. Again, since $v^{(0)}$ and $v$ differ only in indices of the set $I_u$ and the values of coordinates at those indices in $v^{(0)}$ is 0, it is clear that $v^{(0)} \in I_H(v, 0^n) \subseteq I_H(y, 0^n) \subseteq V(G)$. If $I_M$ is an empty set, we denote the vertex $u^{(|I_v|)}$ by $v^{(0)}$.

We continue with the construction of our $u,v$-path by using all the indices from the set $I_u = \{ j_1, j_2, \ldots, j_{|I_u|}\}$. The next vertex of the path, $v^{(1)}$, is obtained from $v^{(0)}$ by replacing the coordinate $v^{(0)}_{j_1}$ with $v_{j_1}$. The vertex $v^{(2)}$, is obtained from $v^{(1)}$ by replacing the coordinate $v^{(1)}_{j_2}$ with $v_{j_2}$. Assume we have already obtained the vertex $v^{(k)}$, then we obtain the vertex $v^{(k+1)}$ from $v^{(k)}$ by replacing the coordinate $v^{(k)}_{j_{k+1}}$ with $v_{j_{k+1}}$. We do this for every index in $I_u$, so the last vertex we obtain is $v^{(|I_u|)}$. It is easy to see, that these vertices indeed form a path (two consecutive vertices differ in exactly one coordinate). Since we only change coordinates, say at index $j'$, from 0 to $v_{j'}$, it is also clear that every vertex constructed in this part of the path belongs to $I_H(v, 0^n) \subseteq I_H(y, 0^n) \subseteq V(G)$. Note, that the vertex $v^{(|I_u|)}$ is actually the vertex $v$. The fact, that the sets $I_M, I_u$ and $I_v$ form a partition of $I_D$ implies that the length of the constructed path is $H(u,v)$. This concludes our proof.
\end{proof}

\section{Characterization of isometric daisy Hamming graphs}\label{sec:dhg}
Let $G'$ be a daisy graph of a Hamming graph $H'=H_{k_1,\ldots , k_{n-1}}$ with respect to $0^{n-1}$. Let $G$ be a peripheral expansion of $G'$ relative to $W_0'=V(G'),W_1',\ldots , W_k'$. If for any $i\in \{1,\ldots ,k\}$, the graph $\langle W_i' \rangle_{H'}$ is a daisy graph of $H'$ with respect to $0^{n-1}$, then the peripheral expansion pe$(G';W_0',\ldots, W_k')$ is called {\em daisy peripheral} expansion of $G'$ relative to  $W_0',\ldots, W_k'$.

In this section we prove that isometric daisy graphs of a Hamming graph are precisely the graphs that can be obtained from $K_1$ by a sequence of daisy peripheral expansions. 

\begin{theorem}\label{p:dexpIsD}
Let $H=H_{k_1,\ldots, k_n}$ be a Hamming graph with the root $0^n$. If $G$ is an isometric daisy graph of $H$ with respect to the root $0^n$, then the daisy peripheral expansion of $G$ relative to the sets $V(G)=W_0,\ldots , W_l$, is an isometric daisy graph of $H'=K_{l+1} \Box H$ with respect to $0^{n+1}$.
\end{theorem}
\begin{proof}
Let $G'$ be the peripheral expansion of $G$ relative to $W_0,W_1,\ldots , W_l$. Therefore, $G'$ consists of a disjoint union of a copy of $G=\left\langle W_0 \right\rangle$ and a copy of $\left\langle W_i \right\rangle$, for any $i\in \{1,\ldots , l\}$. We define the labels of the vertices of $G'$ as follows. Prepend $i$ to each vertex of $G'$ corresponding to the copy of $\left\langle W_i \right\rangle$, for all $i\in \{0,\ldots, l\}$. Hence the labels of the vertices of $G'$ are vectors of length $n+1$ and the first coordinate is an integer from $\{0,\ldots , l\}$.

First, we prove that two vertices of $G'$ are adjacent if and only if the corresponding vectors differ in exactly one position. Since $G'$ is the expansion of $G$ relative to $W_0,\ldots ,W_l$, it follows from the Definition \ref{def:exp} (definition of expansion) that two vertices $u'=u_1\ldots u_n u_{n+1}$ and $v'=v_1\ldots v_n v_{n+1}$ of $G'$ are adjacent in $G'$ if and only if $u=u_2\ldots u_{n+1}$ and $v=v_2\ldots v_{n+1}$ are adjacent in $G$ and both belong to the same set $W_i$, or if $u=u_2\ldots u_{n+1}=v=v_2\ldots v_{n+1}$ and $u$ belongs to two different sets $W_{u_1}$ and $W_{v_1}$. The last condition directly implies that $u'$ and $v'$ differ in exactly one coordinate, namely the first coordinate. If $u=u_2\ldots u_{n+1}$ and $v=v_2\ldots v_{n+1}$ are adjacent in $G$ and contained in the same set $W_i$, then $u$ and $v$ differ in exactly one coordinate. But then, since they are both in $W_i$, $u_1=v_1=i$ and hence $u'$ and $v'$ differ in exactly one coordinate. Hence $G'$ is an induced subgraph of $H'=K_{l+1} \Box H$.

In the second step we prove that $G'$ is a daisy graph of $H'$ with respect to $0^{n+1}$. Let $v'=v_0 v_1\ldots v_n \in V(G')$ and let $x'=x_0\ldots x_n \in I_{H'}(v',0^{n+1})$. Hence $x_i \in \{0,v_i\}$, for any $i \in \{0,\ldots , n\}$. Since $v'=v_0 v_1\ldots v_n$, it follows that $v=v_1\ldots v_n \in W_{v_0}$. We know that the graph $\left\langle W_{v_0} \right\rangle$ is a daisy graph of $H$ with respect to $0^n$ and $x=x_1\ldots x_n \in I_H(v,0^n)$, therefore $x \in V(\left\langle W_{v_0} \right\rangle)$. Hence if  $x'=0x_1\ldots x_n$, then $x'$ is in the copy of $G$ in $G'$. If $x'=v_0 x_1\ldots x_n$, then $x'$ is in the copy of $\left\langle W_{v_0} \right\rangle$ in $G'$. In both cases we deduce that $x' \in V(G')$, which completes this part of the proof.

It remains to prove that $G'$ is an isometric subgraph of $H'$. Let $u'=u_0\ldots u_n$ and $v'=v_0\ldots v_n$ be two arbitrary vertices of $G'$. 

If $u_0=v_0$, then $u=u_1\ldots u_n \in W_{u_0}$ and $v=v_1\ldots v_n \in W_{u_0}$. Since $G'$ is an expansion of $G$, relative to $W_0,\ldots , W_l$, the definition of expansion (Definition \ref{def:exp}) implies that $\left\langle W_{u_0} \right\rangle$ is isometric in $G$. As $G$ is isometric in $H$,
$$d_{\left\langle W_{u_0} \right\rangle}(u,v)=d_G(u,v)=d_H(u,v)=H(u,v).$$
Hence $$d_{G'}(u',v')=d_{\left\langle W_{u_0} \right\rangle}(u,v)=H(u,v)=H(u',v')=d_{H'}(u',v'),$$ where the penultimate equality holds because $u_0=v_0$.

Finally, consider the case where $u_0\neq v_0$. Hence $u=u_1\ldots u_n \in W_{u_0}$ and $v=v_1\ldots v_n \in W_{v_0}$. Since $\left\langle W_{u_0} \cup W_{v_0} \right\rangle$ is isometric in $G$ (by the definition of expansion), there exists a shortest $u,v$-path $P:u=u^0,u^1,\ldots ,u^k=v$ in $G$ (note that each $u^i$ is a vertex in $G$ and hence has the form $u^i=u^i_1\ldots u^i_n$ ) which is entirely contained in $\left\langle W_{u_0} \cup W_{v_0} \right\rangle.$ Since $G$ is isometric in $H$, we get
$$d_{\left\langle W_{u_0} \cup W_{v_0} \right\rangle}(u,v)=d_G(u,v)=d_H(u,v)=H(u,v).$$
Let $i \in \{0,\ldots ,k\}$ be the smallest index such that $u^i \in W_{v_0}$. Since there are no edges between $W_{u_{0}}\setminus W_{v_0}$ and $W_{v_{0}}\setminus W_{u_0}$, $u^i \in W_{u_0}$.  Then the path $u'=u^{0'},u^{1'},\ldots ,u^{i'},v^{i'},\ldots ,v^{k'}=v'$, where $u^{l'}={u_0}u^l$, for any $l\in \{0,\ldots, i\}$ and $v^{l'}={v_0}u^l$, for any $l \in \{i,\ldots , k\}$, is an $u',v'$-path in $G'$. Hence $$d_{G'}(u',v') \leq d_G(u,v)+1=H(u,v)+1=H(u',v')=d_{H'}(u',v'),$$ where the penultimate equality holds because $u_0 \neq v_0$. Since $G'$ is a subgraph of $H'$, the assertion  follows.
\end{proof}

Let $G$ be an isometric daisy graph of a Hamming graph $H=H_{k_1,\ldots , k_n}$ with respect to $0^n$, where $H$ is the smallest possible. We introduce the following terminology which will be used throughout this section. For any $j\in [n]$ we define the sets:
\begin{align*}
    & W_i^j =\{u=u_1\ldots u_n \in V(G)\ |\  u_j=i\}, \text{ for any } i \in [k_j]_0 \\  
    & U_i^j=\{ x \in W_i^j\ |\  \exists y \in W_0^j \land xy \in E(G)  \}, \text{ for any } i\in [k_j]_0 \\  
    & U_{0i}^j=\{ x \in W_0^j\ |\  \exists y \in W_i^j \land xy \in E(G) \}, \text{ for any } i\in \{1, \ldots, k_j-1\} \\
    & U_0^j= \bigcup_{i=1}^{k_j-1}U_{0i}^j. 
\end{align*}
Also, for any $j\in [n]$ and any $i \in [k_j]_0$ denote by $e^j_i$ the vertex of the Hamming graph $H$ labeled by $0^{j-1} i 0^{n-j}$.

\begin{lemma}\label{l:properties of W_ab}
Let $G$ be an isometric daisy graph of a Hamming graph $H=H_{k_1,\ldots , k_n}$ with respect to $0^n$, where $H$ is the smallest possible. For any $j\in [n]$ and any $i\in [k_j]_0$, if $W_i^j \neq \emptyset$, then there exists $uv \in E(G)$ such that $W_i^j = W_{uv}$. 
\end{lemma}

\begin{proof}
Let $j\in [n]$ and $i\in [k_j]_0$ be arbitrary, with $W_i^j \not= \emptyset$, and $x=x_1\ldots x_n \in W_{i}^j$. Hence $x_j=i$. Since $G$ is a daisy graph of $H$ with respect to $0^n$ and $x'=0^{j-1}i0^{n-j} \in I_H(0^n,x)$ it follows that $x'\in V(G)$. Since $x'_j=i$, $x'\in W_i^j$. Then $W_{x'0^n}$ contains exactly all the vertices of $G$, that are closer to $x'$ than $0^n$, i.e. all vertices of $G$ with $j$-th coordinate equal to $i$. Hence $W_{x'0^n}=W_i^j$.
\end{proof}

For the edge $uv$ of a partial Hamming graph, the sets $W_{uv}$ have many nice properties~\cite{Bresar2001,chepoi-88,Wilkeit1990}. Since our graph $G$ is a partial Hamming graph, it follows from Lemma~\ref{l:properties of W_ab} that the sets $W_i^j$ also have these properties.

\begin{lemma}\label{l:0inTriangleClass}
Let $G$ be an isometric daisy graph of a Hamming graph $H=H_{k_1,\ldots , k_n}$ with respect to $0^n$, where $H$ is the smallest possible. For any $\bigtriangleup$-class $F$ of $G$, there exists an edge $f\in F$ with $0^n$ as an endpoint. 
\end{lemma}

\begin{proof}
Let $F$ be an arbitrary $\bigtriangleup$-class of $G$ and $uv \in F$, where $u=u_1 \ldots u_n$ and $v=v_1 \ldots v_n$. Hence, $u_i \not= v_i$, for some $i \in [n]$, and $u_j=v_j$, for any $j \in [n]\setminus \{i\}$. 

First, suppose that one of $u_i$ and $v_i$ equals $0$, say $u_i$. It follows that $0^n \in W_{uv}$. Since $e^i_{v_i} \in I_H(v, 0^n)$ and $G$ is a daisy graph of $H$ with respect to $0^n$, it follows that $e^i_{v_i} \in V(G)$. Since the $i^\text{th}$ coordinate of $e^i_{v_i}$ is $v_i$, the vertex $e^i_{v_i} \in W_{vu}$. Hence, $0^n e^i_{v_i} \sim uv$ and therefore $0^n e^i_{v_i} \in F$.

Finally, suppose neither $u_i$ nor $v_i$ equals $0$. Since $x=u_1 \ldots u_{i-1} 0 u_{i+1} \ldots u_n \in I_H(u,0^n)$ and $G$ is a daisy graph of $H$ with respect to $0^n$, the vertex $x \in V(G)$. Note that $u,v \text{ and } x$ induce $K_3$ in $G$. Hence, $vx \bigtriangleup uv$ and consequently the edge $vx$ belongs to $F$. Now, consider the vertex $e^i_{v_i}$, which belongs to $I_H(v,0^n)$ and therefore is a vertex of $G$. Similarly to the first case, we deduce that $e^i_{v_i} \in W_{vx}$. Clearly, $0^n \in W_{xv}$ and $0^n e^i_{v_i}$ is an edge of $G$. It follows that $0^n e^i_{v_i} \sim xv$ and therefore $0^n e^i_{v_i} \in F$.
\end{proof}

From the definition of the relation $\bigtriangleup$ it follows that the $\bigtriangleup$-class $F_j$ generated by the edge $0^n e^j_i$, for some $i\not = 0$, contains exactly all edges between $U_k^j$ and $U_l^j$, for any $0 \leq k < l \leq k_j-1$. Thus using Lemma~\ref{l:0inTriangleClass} we deduce the following. 

\begin{cor}\label{c:TiangleClasses} Let $G$ be an isometric daisy graph of a Hamming graph $H=H_{k_1,\ldots , k_n}$ with respect to $0^n$, where $H$ is the smallest possible. There are exactly $n$ $\bigtriangleup$-classes $F_1,\ldots , F_n$ of $E(G)$, where for any $j\in [n]$ the $\bigtriangleup$-class $F_j$ is generated by the edge $0^n e^j_i$, for some $0 < i \leq k_j-1$.
\end{cor}

Let $G$ be an isometric daisy graph of a Hamming graph $H=H_{k_1,\ldots , k_n}$ with respect to $0^n$, where $H$ is the smallest possible. Let $j\in [n]$ and $i\in [k_j]_0$. A subgraph $\langle W_i^j \rangle$ of a graph $G$ is called {\em peripheral} if $U_i^j=W_i^j$. The $\bigtriangleup$-class $F$ generated by the edge $0^n e^j_l$, for some $0 < l \leq k_j-1$, of the graph $G$ is called {\em peripheral} if $U_{l'}^j=W_{l'}^j$, for any $l' \in \{1,\ldots ,k_j-1\}$.

\begin{lemma}\label{l:peripheral}
If $G$ is an isometric daisy graph of a Hamming graph $H=H_{k_1,\ldots , k_n}$ with respect to $0^n$, where $H$ is the smallest possible, then every $\bigtriangleup$-class $F$ of the graph $G$ is peripheral. 
\end{lemma}
\begin{proof}
Let $F$ be an arbitrarily chosen $\bigtriangleup$-class of $G$, such that $0^n e^j_l \in F$. Let $i \in \{1,\ldots, k_j-1\}$ be arbitrary. To prove the assertion, we will show that any vertex of $W^j_i$ has a neighbour in $W^j_0$ (which means $W^j_i = U^j_i$). Take any $x = x_1 \ldots x_n \in W^j_i$, hence $x_j = i$. Now, consider $x' = x_1 \ldots x_{j-1} 0 x_{j+1} \ldots x_n$. Note, that $x' \in I_H(0^n,x) \subseteq V(G)$ and therefore $x' \in W^j_0$. Since $xx' \in E(G)$, the assertion follows.
\end{proof}

\begin{lemma}\label{l:daisySubgraphs}
Let $G$ be an isometric daisy graph of a Hamming graph $H=H_{k_1,\ldots , k_n}$ with respect to $0^n$, where $H$ is the smallest possible. For every $j\in[n]$ and any $i\in [k_j]_0$ the subgraph $\langle W_i^j \rangle$ of the graph $G$ is a daisy graph of $H'=H_{k_1,\ldots ,k_{j-1},k_{j+1},\ldots,k_n}$ with respect to $0^{n-1}$. 
\end{lemma}
\begin{proof}
Define $X^j_i=\{ x_1 \ldots x_{j-1} x_{j+1} \ldots x_n \ |\ x_1\ldots x_n \in W^j_i\}$. Let $r:W^j_i \rightarrow X^j_i$ be the projection defined by $r: x_1\ldots x_n \mapsto x_1 \ldots x_{j-1} x_{j+1} \ldots x_n$, which is clearly bijection between $W^j_i$ and $X^j_i$. 

Let $u=u_1 \ldots u_{n-1} \in X^j_i$ be arbitrary and $w \in I_{H'}(0^{n-1},u)$. We claim that $w \in X^j_i$. Since $u \in X^j_i$, it follows from the definition of $X^j_i$ that $u'=u_1\ldots u_{j-1} i u_j\ldots u_{n-1} \in W^j_i$. Since $w \in I_{H'}(0^{n-1},u)$, it follows that $w_l = u_l$ or $w_l = 0$, for all $1\leq l \leq n-1$. Let $w'=w_1 \ldots w_{j-1} i w_{j} \ldots w_{n-1}$. Since $w' \in I_H(0^n,u')$, it follows that $w' \in V(G)$ and as the $i^{\textrm{th}}$ coordinate of $w'$ is $i$, the vertex $w'$ belongs to $W^j_i.$ By the definition of $X^j_i$, $w \in X^j_i$. Therefore $\langle X^j_i \rangle _{H'}$ is a daisy graph of $H'$ with respect to $0^{n-1}$. Since $\langle W^j_i \rangle _H \cong \langle X^j_i \rangle_{H'}$, the assertion follows.
\end{proof}

In~\cite{Bresar2001} the contraction of a partial Hamming graph $G$ was defined in the following way. Let $uv \in E(G)$ and let $\bigtriangleup$-class with respect to $uv \in E(G)$, denote it by  $\bigtriangleup_{uv}$, be the union of $k$ distinct $\sim$-classes $F_{x_ix_j}$. A graph $G'$ is a contraction of a partial Hamming graph $G$ with respect to the edge $uv \in E(G)$ if each clique induced by edges belonging to $\bigtriangleup_{uv}$ is contracted to a single vertex. For all $i \in [k]$, let $W_i'$ be the set of vertices in $G'$ that corresponds to $W_{x_i}=\{w\in V(G)\ |\ d(w,x_i) < d(w, x_j), \text{ for any }j\ne i\}$. Bre\v sar proved that the expansion of $G'$ relative to $W_1',\ldots , W_k'$ is exactly the graph $G$~\cite{Bresar2001}. 

\begin{theorem}\label{th:contraction}
Let $G$ be an isometric daisy graph of a graph $H=H_{k_1,\ldots, k_n}$ with respect to $0^n$, where $H$ is the smallest possible. Then there exists a daisy graph $G' \subseteq G$ such that $G$ can be obtained from $G'$ by a daisy peripheral expansion.
\end{theorem}
\begin{proof}
Let $F$ be an arbitrary $\bigtriangleup$-class of the graph $G$. By  Corollary \ref{c:TiangleClasses} there exist $j \in [n]$ and $i \in \{1,\ldots , k_j-1\}$ such that $F$ is generated by the edge $0^n e^j_i$. Let the graph $G'$ be obtained from the graph $G$ by a contraction with respect to the edge $0^n e^j_i$. For any $l \in [k_j]_0$, denote by $X_l$ the set of vertices in $G'$ that corresponds to $W^j_l$ in $G$. By the definition of a contraction, the graph $G$ is the expansion of $G'$ relative to sets $X_0,\ldots , X_{k_j-1}$. By Lemma \ref{l:peripheral} it follows that $F$ is a peripheral $\bigtriangleup$-class. Using the fact that $F$ is generated by the edge $0^n e^j_i$, it follows from the definition of peripheral classes, that $U^j_{i'}=W^j_{i'}$, for any $i'\in \{1, \ldots, k_j-1\}$ (every vertex of $W^j_{i'}$ has a neighbour in $W^j_0$). Since $\bigcup_{i=0}^{k_j-i} X_i = V(G)$ (definition of expansion) we obtain that $X_0 = V(G')$. By Lemma \ref{l:daisySubgraphs} it follows that the subgraphs $\langle X_i \rangle_{G'}$ are daisy graphs which proves that $G$ is obtained from $G'$ by daisy peripheral expansion. 
\end{proof}

From Theorem \ref{p:dexpIsD} and Theorem \ref{th:contraction} we immediately obtain the following characterization.

\begin{theorem}
A graph $G$ is an isometric daisy graph of a graph $H=H_{k_1,\ldots , k_n}$ with respect to $0^n$, where $H$ is the smallest possible, if and only if it can be obtained from the one vertex graph by a sequence of daisy peripheral expansions.
\end{theorem}

\subsection*{Acknowledgments}
This work was supported by the Slovenian Research Agency under the grants P1-0297, J1-1693 and J1-9109.

\end{document}